\documentclass[reqno, oneside, 12pt]{amsart}

\usepackage[letterpaper]{geometry}
\usepackage{array}

\usepackage{enumerate, hyperref,url,amssymb,amsmath,amsthm,amsxtra,mathtools,mathrsfs,calc,nccmath,color}

\usepackage{calc}
\usepackage{graphicx}
\usepackage{caption}
\usepackage{subcaption}

\newcommand{\Z}{\mathbb{Z}}
\newcommand{\Q}{\mathbb{Q}}

\newcommand{\C}{\mathbb{C}}

\let\temp\phi
\let\phi\varphi
\let\varphi\temp

\renewcommand{\(}{\left(}
\renewcommand{\)}{\right)}

\newcommand{\GL}{\operatorname{GL}}

\renewcommand{\sl}{\big|}
\newcommand{\sk}{\big|_k }

\renewcommand{\bar}[1]{\overline{#1}}

\newtheorem{theorem}{Theorem}[section]
\newtheorem{lemma}[theorem]{Lemma}

\newtheorem{proposition}[theorem]{Proposition}

\theoremstyle{remark}

\numberwithin{equation}{section}

\begin{document}

%%%%%%%%%%%%%%%%%%%%%%%%%%%%%%%%%%%%%%%%%
%               Title, Etc.             %
%%%%%%%%%%%%%%%%%%%%%%%%%%%%%%%%%%%%%%%%%

\title{Weight $2$ CM newforms as p-adic limits}

\date{\today}
\author{Robert Dicks}
\address{Department of Mathematics\\
University of Illinois\\
Urbana, IL 61801} 
\email{rdicks2@illinois.edu}

%%%%%%%%%%%%%%%%%%%%%%%%%%%%%%%
%          Abstract           %
%%%%%%%%%%%%%%%%%%%%%%%%%%%%%%%
 
\begin{abstract}
Previous works have shown that certain weight $2$ newforms are $p$-adic limits of weakly holomorphic modular forms under repeated application of the $U$-operator. The proofs of these theorems originally relied on the theory of harmonic Maass forms. Ahlgren and Samart obtained strengthened versions of these results using the theory of holomorphic modular forms. Here, we use such techniques to express all weight $2$ CM newforms which are eta quotients as $p$-adic limits. In particular, we show that these forms are $p$-adic limits of the derivatives of the Weierstrass mock modular forms associated to their elliptic curves.
\end{abstract}

%%%%%%%%%%%%%%%%%%%%%%%%%%%%%%%%%%%%%%%%%
%               Document Text           %
%%%%%%%%%%%%%%%%%%%%%%%%%%%%%%%%%%%%%%%%%
 \maketitle

%%%%%%%%%%%%%%%%%%%%%%%%%%%%%%%%%%%%%%%%%%%%%%%%%%%%%%%%%%%%%%%%%%%%%%%%%%%%%%%%
% Introduction
%%%%%%%%%%%%%%%%%%%%%%%%%%%%%%%%%%%%%%%%%%%%%%%%%%%%%%%%%%%%%%%%%%%%%%%%%%%%%%%%
 
 \section{Introduction}
Suppose that $E$ is an elliptic curve over $\Q$ with Weierstrass model
\[
E: y^{2}+a_{1}xy+a_{3}y=x^{3}+a_{2}x^{2}+a_{4}x+a_{6}.
\]
Let $N_{E}$ be its conductor. By the modularity of elliptic curves over $\Q$, there exists a lattice $\Lambda_{E}$ and a modular parametrization

\[
\Phi_{E}:X_{0}(N_{E}) \rightarrow \C/\Lambda_{E} \cong E.
\]

Martin and Ono \cite{Martin-Ono} proved that there are five such curves with complex multiplication whose associated newform, $g_{N_{E}}(z)$, is an eta-quotient. These forms lie in $S_{2}(N_{E})$, the space of weight $2$ cusp forms on $\Gamma_0(N_{E})$. The following table lists these curves.
 \begin{table}[h!]
\centering
 \begin{tabular}{  |c |c |c c c c c c |c|  }
 \hline
 $N_{E}$ & $g_{N_{E}}(z)$ & $a_{1}$ & $a_{2}$ & $a_{3}$ & $a_{4}$ & $a_{6}$ & CM field\\ 
\hline
 27 &  $\eta^{2}(3z)\eta^{2}(9z)$ & 0 & 0&1&0&-7&$\Q(\sqrt{-3})$\\ 
 
 32 & $\eta^{2}(4z)\eta^{2}(8z)$ & 0 & 0&0&4&0&$\Q(i)$\\
 
 36 & $\eta^{4}(6z)$ & 0 & 0&0&0&1&$\Q(\sqrt{-3}$)\\
 
 64 & $\frac{\eta^{8}(8z)}{\eta^{2}(4z)\eta^{2}(16z)}$ &0&0&0&-4&0&$\Q(i)$\\
 
 144 & $\frac{\eta^{12}(12z)}{\eta^{4}(6z)\eta^{4}(24z)}$ & 0 & 0&0&0&-1&$\Q(\sqrt{-3})$\\  
 \hline
\end{tabular}
\caption{The five curves}
\end{table}

For an elliptic curve $E$ over $\Q$, Alfes, Griffin, Ono, and Rolen \cite{Weierstrass} constructed harmonic Maass forms using the Weierstrass $\zeta$-function which comes from its modular parametrization. These forms encode arithmetic information about its Hasse-Weil $L$-function. The holomorphic parts of these forms are called \textit{Weierstrass mock modular forms}.
For each of the above curves, Clemm \cite{Clemm} showed that the derivative of the associated Weierstrass mock modular form is an eta-quotient or a twist of one. 
The following table lists these forms (where $\chi_{8}$ and $\chi_{12}$ are the Kronecker symbols with discriminants $8$ and $12$). We have normalized them to have leading coefficient $1$.
\begin{center}
 \begin{tabular}{  |c|             |c|  } 
 \hline
 $N_{E}$ & \    \ $G_{N_{E}}(z)$\\ 
\hline 

 27 &          \     \  $\frac{\eta(3z)\eta^{6}(9z)}{\eta^{3}(27z)}$ \\ 
 
 32 &          \        \  $\frac{\eta^{2}(4z)\eta^{6}(16z)}{\eta^{4}(32z)}$ \\
 
 36 &          \        \ $\frac{\eta^{3}(6z)\eta(12z)\eta^{3}(18z)}{\eta^{3}(36z)}$ \\
 
 64 &          \        \ $G_{32} \otimes \chi_{8}$ \\
 
 144 &         \       \ $G_{36} \otimes \chi_{12}$\\  
 \hline
\end{tabular}{}
\end{center}
Write $g_{N_{E}}=\displaystyle \sum a_{N_{E}}(n)q^{n}$ and $G_{N_{E}}=\displaystyle \sum C_{N_{E}}(n)q^{n}$. Using the theory of harmonic Maass forms, Clemm showed that if $p$ is a prime which is inert in the CM field of $E$ and $p \nmid C_{N_{E}}(p)$, then as a $p$-adic limit, we have
\begin{equation}\label{Clemm}
\lim_{m \rightarrow \infty}\frac{G_{N_{E}} \sl U(p^{2m+1})}{C_{N_{E}}(p^{2m+1})}=g_{N_{E}}.
\end{equation}
(She omits the hypothesis that $p \nmid C_{N_{E}}(p)$ in \cite{Clemm}).
In the case where $N_{E}=32$, the cusp form $g_{32}$ gives the Hasse-Weil $L$-function of the congruent number elliptic curve
\[
y^{2}=x^{3}-x.
\]
 For primes $p$ which are inert in $\Q(i)$ with $p \nmid C_{32}(p)$, El-Guindy and Ono \cite{El-Guindy-Ono} showed \eqref{Clemm} for $N_{E}=32$ using the theory of harmonic Maass forms. 
Ahlgren and Samart strengthened this result using the theory of holomorphic modular forms. 
Let $v_{p}(\cdot)$ denote the $p$-adic valuation on $\Z[[q]]$. 
If $p$ is a prime which is inert in $\Q(i)$, they show that for all integers $m \geq 0$, we have

 \[
v_{p}(C_{32}(p^{2m+1}))=m,
\]

 \[
v_{p}\(  \frac{G_{32} \sl U(p^{2m+1})}{C_{32}(p^{2m+1})}-g_{32}  \) \geq m+1.
\]
In this paper, we prove the following theorem.

\begin{theorem}\label{thm:main}
 Let $E$ be one of the elliptic curves in Table $1$. Write $g_{N_{E}}=\displaystyle \sum a_{N_{E}}(n)q^{n}$ and  $G_{N_{E}}=\displaystyle \sum C_{N_{E}}(n)q^{n}$. Let $p$ be a prime which is inert in the CM field of $E$.
  Then for all integers $m \geq 0$, we have
 \begin{equation}\label{thm:main part 1}
 v_{p}(C_{N_{E}}(p^{2m+1}))=m,
 \end{equation}
 
  \begin{equation}\label{thm:main part 2}
v_{p}\(  \frac{G_{N_{E}} \sl U(p^{2m+1})}{C_{N_{E}}(p^{2m+1})}-g_{N_{E}}  \) \geq m+1.
\end{equation}
 \end{theorem}
The paper is organized as follows. In Section $2$, we give background material on modular forms. In Sections $3$, $4$, and $5$, we prove Theorem~\ref{thm:main}.
 \section{Background}
 For the next several paragraphs, we follow the exposition in \cite{Ahlgren-Samart}.
 Suppose that $k \in \mathbb{Z}$ and that $N$ is a positive integer. For a function $f(z)$ on the upper half plane and 
\[
 \gamma =\left(\begin{matrix}a & b \\c & d\end{matrix}\right) \in \GL_{2}^{+}(\Q),
\]
we have the weight $k$ slash operator 
\[
f(z)\sk \gamma := \det(\gamma)^{\frac{k}{2}}(cz+d)^{-k}f\(\frac{az+b}{cz+d}\).
\] 
We denote by $M^{!}_{k}(N)$ the space of forms $f$ which satisfy the transformation law

\[
f \sk \gamma= f \ \ \ \text{ for } \ \ \ \gamma = \left(\begin{matrix}a & b \\c & d\end{matrix}\right) \in \Gamma_0(N)
\]
and which are holomorphic on the upper half plane and meromorphic at the cusps of $\Gamma_0(N)$.
We identify each $f \in M_{k}^{!}(N)$ with its $q$-expansion. 
That is, if $q:=\exp(2\pi i z)$, we can write $f(z)= \sum a(n)q^{n}$ for some coefficients $a(n)$. 
Let $M_k(N) \subseteq M_k^{!}(N)$ be the subspace of forms which are holomorphic at the cusps of $\Gamma_0(N)$.
Let $M_k^{\infty}(N) \subseteq M_k^{!}(N)$ be the subspace of forms which vanish at all  cusps of $\Gamma_0(N)$ other than $\infty$ and $S_k(N) \subseteq M_k^{\infty}(N)$ be the subspace of forms which vanish at all of the cusps.

We next recall the $U$ and $V$ operators. For a positive integer $m$, we define them on Fourier expansions by
\[
\(\sum a(n)q^{n}\)\sl U_{m}:= \sum a(mn)q^{n},
\]

\[
\(\sum a(n)q^{n}\)\sl V_{m}:= \sum a(n)q^{mn}.
\]
For a positive integer $m$, let $T_k(m)$ be the usual Hecke operator on $M^{!}_{k}(N)$. If $p$ is a prime with $p \nmid N$, $n \geq 1$ is an integer, and $f \in M_{k}^{!}(N)$, then
 \begin{equation}\label{Heckeoperator}
 F\sl T_{k}(p^n)=\sum_{j=0}^{n}p^{(k-1)j}f\sl U(p^{n-j})\sl V(p^{j}).
 \end{equation}
Define
\begin{equation}\label{RamanujanTheta}
\Theta:=\frac{1}{2\pi i}\frac{d}{dz}=q\frac{d}{dq}.
\end{equation}
 We have the following result (see e.g. \cite[Lemma $2.1$]{Ahlgren-Samart}).
% An argument of Ahlgren and Samart (see \cite[Lemma $2.1$]{Ahlgren-Samart}) implies the following result.
 \begin{lemma}\label{Heckelemma}
 If $(m,N)=1$, then 
 \begin{equation}\label{Hecke}
 T_{k}(m): M^{\infty}_{k}(N) \rightarrow M^{\infty}_{k}(N).
  \end{equation}
   If $k \geq 2$, then 
  \begin{equation}\label{iteratedTheta}
  \Theta^{k-1}:M^{\infty}_{2-k}(N) \rightarrow M^{\infty}_{k}(N). 
  \end{equation}
 \end{lemma}
 
 Finally, we review some facts about filtrations. If $p$ is a prime such that $p \nmid 6N$ and $k$ is a nonnegative integer, let $M^{(p)}_{k}(N)$ be the subset of forms in $M_{k}(N)$ which have coefficients which are $p$-integral rational numbers. 
 if $f=\sum a(n)q^{n} \in M^{(p)}_{k}(N)$, 
we define
\[
 \bar{f}:= \sum \bar{a(n)}q^{n} \in \mathbb{F}_{p}[[q]], 
 \]
 and we define 

\[
\bar{M^{(p)}_{k}(N)}:= \{\bar{f}: \ \ f \in M^{(p)}_{k}(N)\}.
\]
%If $k$ is an integer and $f \in M_{k}(N)$, then we define the filtration of $f$ to be 
%\[
%\omega(f)=\omega(\bar{f}):= \inf\{k': \ \text{there exists } g \in M_{k'}(N)\text{ with } \bar{f}=\bar{g}\}.
%\]
 If 
 $f \in M^{(p)}_{k}(N)$,  then we define the filtration of $\bar{f}$ as
\[
w_{p}(\bar{f})=\inf\{k': \ \text{there exists } g \in M_{k'}(N)\text{ with }  \bar{f}=\bar{g}\}.
\]
We make use of the following facts from \cite[~$\mathsection$$7$]{Jochnowitz}. 

\begin{proposition}\label{Jochnowitz}
Suppose that $f \in M^{(p)}_{k}(N)$ and that $w_{p}(\bar{f}) \neq -\infty$. Then we have
\begin{enumerate}
\item
$w_{p}(\bar{f}) \equiv k \pmod{p-1}$.
\item
$w_{p}\(\bar{f\sl V(p)}\)=pw_{p}(\bar{f}).$
\end{enumerate}
\end{proposition}
 \section{Proof of Theorem~\ref{thm:main} for $N_{E}=27$}
The proof of Theorem~\ref{thm:main} for $N_{E}=27$ requires proving some preliminary results. Recalling our notation $g_{27}=\displaystyle \sum a_{27}(n)q^{n}$ and $G_{27}=\displaystyle \sum C_{27}(n)q^{n}$, we have the following result.
 
 \begin{proposition}\label{proposition3.1}
 For every integer $m \geq -1$ except for $m=0$, there exists a unique form $H_{m} \in M^{\infty}_{2}(27) \cap \Z((q))$ of the form
 
 \[
 H_{m}=q^{-m}+O(q^{2}).
 \]
 Moreover, if $p \neq 3$ is prime and $n$ is a nonnegative integer, we have
 
 \[
 G_{27} \sl T_{2}(p^{n})=p^{n}H_{p^{n}}+C_{27}(p^{n})g_{27}.
 \]
 \end{proposition}
 
 \begin{proof}
 Define
 
\begin{align*}
L_{1}(z)&=\frac{\eta^{4}(9z)}{\eta(3z)\eta^{3}(27z)}=q^{-2}+q+O(q^{4}),\\
L_{2}(z)&=\frac{\eta^{3}(3z)}{\eta^{3}(27z)}=q^{-3}-3+O(q^{3}).
\end{align*}
  For integers $d \geq 0$, consider the forms
 
 \begin{equation}\label{3.1}
g_{27}(z)L_{1}^{d}(z)=q^{-2d+1}+O(q^{-2d+4}),
 \end{equation}
 \begin{equation}\label{3.2}
 g_{27}(z)L^{d}_{1}(z)L_{2}(z)=q^{-2d}+O(q^{-2d+3}).
 \end{equation}
 The fact that the forms $L_{1}$ and $L_{2}$ are holomorphic at all of the cusps other than $\infty$ (see e.g. \cite[Thm $1.64$, Thm $1.65$]{Ken}) implies that the forms in \eqref{3.1} and \eqref{3.2} are in $M^{\infty}_{2}(27)$. By taking linear combinations of these forms, we obtain the forms $H_{m}$.
Since the space $S_{2}(27)$ is one-dimensional, the forms $H_{m}$ are unique. Some examples of these forms are given below.
  \begin{align*}
  &H_{-1}=g_{27}=q-2q^{4}+\cdots,\\
  &H_{1}=G_{27}=q^{-1}-q^{2}+\cdots,\\
  &H_{2}=q^{-2}-5q^{4}+\cdots.
  \end{align*}
   To prove the last part of the proposition, note that 
 
 \[
 G_{27} \sl T_2(p^n)=G_{27} \sl U(p^n)+\sum_{j=1}^{n-1}p^{j}G_{27}\sl U(p^{n-j})\sl V(p^{j})+p^{n}G_{27}\sl V(p^{n})
 \]
 by \eqref{Heckeoperator}. By \eqref{Hecke}, we have $G_{27}\sl T_2(p^n) \in M_{2}^{\infty}(27)$. Since
 \[
 G_{27}=q^{-1}-q^{2}+\cdots,
 \]
we have
 \[
 G_{27} \sl U(p^n)=C_{27}(p^n)g_{27}+O(q^{2})
 \]
 and
 \[
 \sum_{j=1}^{n-1}p^{j}G_{27}\sl U(p^{n-j})\sl V(p^{j})+p^{n}G_{27}\sl V(p^n)=p^{n}q^{-p^{n}}+O(q^{2}).
 \]
This implies that
 \[
 G_{27}\sl T_2(p^n)-p^nH_{p^n{}}
 \in S_2(27).
 \]
 Since $S_2(27)$ is one-dimensional, we have $G_{27}\sl T_2(p^n)-p^{n}H_{p^{n}}=C_{27}(p^n)g_{27}$.
  \end{proof}
 
 Before we prove Theorem~\ref{thm:main} for $N_{E}=27$, we require the following congruence.
 \begin{lemma}\label{valuationequalsm}
 For each prime $p$ which is inert in $\Q(\sqrt{-3})$ and each integer $m \geq 0$, we have the congruence
 \[
 C_{27}(p^{2m+1}) \equiv (-1)^{m}p^{m}C_{27}(p) \pmod{p^{m+1}}.
 \]
 \end{lemma}
 
 \begin{proof}
 By definition, $L_{1}$ has an expansion of the form
 \[
 L_{1}(z)=\sum b(n)q^{3n+1} \in \Z((q)),
 \]
and $L_{2}$ has an expansion of the form
 \[
 L_{2}(z)=\sum c(n)q^{3n} \in \Z((q)).
 \]
Suppose that $p \equiv 2 \pmod{3}$.
By taking polynomials in $L_{1}$ and $L_{2}$ with integer coefficients, we can construct a modular function $\psi_{p} \in M^{\!}_{0}(27)$ which is holomorphic at all of the cusps of $\Gamma_0(27)$ other than $\infty$ of the form
 \begin{equation}\label{3.3}
 \psi_{p}(z)=q^{-p}+C_{p}q+O(q^{4}).
 \end{equation}

Since $G_{27} \in M^{\infty}_{2}(27)$, we see that the meromorphic differential $G_{27}(z)\psi_{p}(z)dz$ on the modular curve $X_0(27)$ is holomorphic at all of the cusps of $X_0(27)$ other than $\infty$. The sum of the residues of a meromorphic differential is zero, and the residue at $\infty$ is a multiple of the constant term. Since the constant term of $G(z)\psi_{p}(z)$ is $C_{p}+C_{27}(p)$, we have $C_{p}=-C_{27}(p)$.
 
 By Lemma~\ref{Heckelemma}, we have
 
 \begin{equation}\label{3.4}
 \Theta(\psi_{p})=-pq^{-p}-C_{27}(p)q+O(q^{4}) \in M^{\infty}_{2}(27).
 \end{equation}
 By Proposition~\ref{proposition3.1}, we also have
 \[
 -G_{27}\sl T_{2}(p)=-pq^{-p}-C_{27}(p)q+O(q^{4}).
 \]
 Thus,
 \begin{equation}\label{whywasthisnottagged}
 G_{27} \sl T_{2}(p)=-\Theta(\psi_{p}),
 \end{equation}
 which means that
 \begin{equation}\label{3.5}
 G_{27} \sl U(p)=-\Theta(\psi_{p})-pG_{27}\sl V(p).
 \end{equation}
 By applying $U(p^2)$ to both sides of \eqref{3.5}, an induction argument implies for each $m \geq 0$ that
 \[
 G_{27} \sl U(p^{2m+1})=\sum_{k=0}^{m}(-1)^{m-k+1}p^{m-k}\Theta(\psi_{p})\sl U(p^{2k})+(-1)^{m+1}p^{m+1}G_{27}\sl V(p).
 \]
 The fact that  $\Theta(\psi_{p})\sl U(p^{2k}) \equiv 0 \pmod{p^{2k}}$ implies for each integer $m \geq 0$ that
 \begin{equation}\label{anunexpectedlylongname}
 G_{27} \sl U(p^{2m+1}) \equiv (-1)^{m+1}p^{m}\Theta(\psi_{p}) \pmod{p^{m+1}}.
 \end{equation}
The result follows by comparing the coefficients of $q$ in \eqref{anunexpectedlylongname} using \eqref{3.4}.
 \end{proof}
 
 The next lemma helps us to establish \eqref{thm:main part 1}.
 
 \begin{lemma}\label{lemma3.3}
 For each prime $p$ which is inert in $\Q(\sqrt{-3})$, we have $p \nmid C_{27}(p)$.
 \end{lemma}
 
 \begin{proof}
 %We follow the argument in \cite{Ahlgren-Samart}. 
 Assume for the sake of a contradiction that $p \mid C_{27}(p)$. From  Proposition~\ref{proposition3.1} and \eqref{whywasthisnottagged}, we see that
 \[
 \Theta(\psi_{p}) =-G_{27} \sl T_{2}(p) = -pH_{p}-C_{27}(p)g_{27} \equiv 0 \pmod{p}.
 \]
 This implies that $\psi_{p}$ is of the form
 \[
 \psi_{p} \equiv q^{-p}+\sum_{n=1}^{\infty}A(pn)q^{pn} \pmod{p}
 \]
 for some integral coefficients $A(pn)$.
 A computation in Magma of an integral basis for $M_{2}(27)$ gives
 \begin{equation}\label{basis27}
 f_{1}=q+\cdots, \ \
 f_{2}=q^{2}+\cdots, \ \
 f_{3}=q^{3}+\cdots, \ \
 f_{4}=q^{4}+\cdots, \ \
 f_{5}=q^{6}+\cdots.
 \end{equation}
% From a Magma computation, 
%From this basis,
 %we see that there is a form $f \in M_{2}(27)$ with
% \[
% f = q^{6}+\cdots.
 %\]
 %This implies that $f_5^p$ is of the form
% \[
% f_5^{p} \equiv \sum_{n=6}^{\infty}B(pn)q^{pn} \equiv q^{6p}+\cdots \pmod{p}.
% \]
 Set
 $h_{p}=\psi_{p}f_5^{p} \in M_{2p}(27).$
 For some integral coefficients $D(pn)$, we have
 \[
 h_{p} \equiv \sum_{n=5}^{\infty}D(pn)q^{pn} \equiv q^{5p}+\cdots \pmod{p},
 \]
 which means that
 \[
 h_{p} \equiv h_{p}\sl U(p) \sl V(p) \pmod{p}.
 \]
 Since
 \[
 h_{p}\sl U(p) \equiv h_{p} \sl T(p) \pmod{p},
 \]
 we have $\bar{h_{p}\sl U(p)} \in \bar{M^{(p)}_{2}(27)}$.
 By $(2)$ of Proposition~\ref{Jochnowitz}, we have
 $w_p(\bar{h_{p}})=pw_p(\bar{h_{p}\sl U(p)})$
 and $w_p(\bar{h_{p}}) \equiv 2p \pmod{p-1}$. The fact that $p \mid w_p(\bar{h_{p}})$ implies that $w_p(\bar{h_{p}})=2p$, so we have $w_p(\bar{h_{p}\sl U(p)})=2$. Thus, there exists a form $h_0 \in M_{2}^{(p)}(27)$ with
 \[
 h_{0} \equiv h_{p}\sl U(p) \equiv q^{5}+O(q^{6}) \pmod{p}.
 \]
However, an examination of the above basis for $M_2(27)$ shows that no such form $h_{0}$ exists. The result follows.
 \end{proof}
 \begin{proof}[Proof of Theorem~\ref{thm:main} for $N_{E}=27$] Lemmas $3.2$ and $3.3$ imply \eqref{thm:main part 1}. To prove \eqref{thm:main part 2}, note that Proposition $3.1$ and \eqref{Heckeoperator} give
 \begin{equation}\label{3.7}
 \frac{G_{27}\sl U(p^{2m+1})}{C_{27}(p^{2m+1})}-g_{27}=\frac{1}{C_{27}(p^{2m+1})}\(p^{2m+1}H_{p^{2m+1}}-\sum_{j=1}^{2m+1}p^{j}G_{27}\sl U(p^{2m+1-j})\sl V(p^{j})\).
 \end{equation}
 
 By \eqref{Heckeoperator}, we have
 \[
 G_{27}\sl T_{2}(p^{2m})= \sum_{j=1}^{2m+1}p^{j-1}G_{27}\sl U(p^{2m+1-j})\sl V(p^{j-1}).
 \]
 By the definition of $G_{27}$, we see that $C_{27}(n)=0$ whenever $n \not \equiv 2 \pmod{3}$. Proposition~\ref{proposition3.1} then implies that $G_{27} \sl T_{2}(p^{2m})=p^{2m}H_{p^{2m}}$. Thus, we have
 \begin{equation}\label{3.8}
 \sum_{j=1}^{2m+1}p^{j}G_{27}\sl U(p^{2m+1-j})\sl U(p^{2m+1-j})\sl V(p^{j})=p^{2m+1}H_{p^{2m}}\sl V(p) \equiv 0 \pmod{p^{2m+1}}.
 \end{equation}
 From \eqref{3.8}, we see that \eqref{thm:main part 1} and \eqref{3.7} give \eqref{thm:main part 2}.
 \end{proof}
  \section{Proof of Theorem~\ref{thm:main} for $N_{E}=36$}
  The proof of Theorem~\ref{thm:main} for $N_{E}=36$ is similar to the proof for $N_{E}=27$; fewer details will be given. We have the following analogue of Proposition~\ref{proposition3.1} above.
  
  \begin{proposition}\label{proposition4.1}
  Suppose that $m \geq -1$ is an odd integer. There exists a unique form $H_{m}\in M_{2}^{\infty}(36) \cap \Z((q))$ of the form
  \[
  H_{m}=q^{-m}+O(q^{3}).
  \]
  For nonnegative integers $n$ and primes $p \geq 5$, we have
  \[
  G_{36}\sl T_{2}(p^{n})=p^{n}H_{p^{n}}+C_{36}(p^{n})g_{36}.
  \]
  \end{proposition}
  
  \begin{proof}
  Define
  \[
  L(z):=\frac{\eta(6z)\eta^{3}(9z)}{\eta(3z)\eta^{3}(18z)}.
  \]
 % As in the proof of Proposition~\ref{proposition3.1},
 For integers $d \geq 0$, consider the forms
  \begin{equation}\label{anextraordinarilylongwaytosaynothingatall}
 g_{36}(z)L^{d}(2z)=q^{-2d+1}+O(q^{-2d+7}).
  \end{equation}
  %By Theorems $1.64$ and $1.65$ of \cite{Ken}, we again have $F_{d} \in M_{2}(36)$. 
  Since $L(2z)$ is holomorphic at all of the cusps of $\Gamma_0(36)$ other than $\infty$, the forms in 
  \eqref{anextraordinarilylongwaytosaynothingatall} are in $M^{\infty}_{2}(36)$. These forms are supported on odd exponents; it follows that we can take linear combinations to get forms $H_{m}$ with the desired properties. These forms are unique because the space $S_{2}(36)$ is one-dimensional. The proof of the last assertion follows as in the proof of Proposition~\ref{proposition3.1}.
  \end{proof}
  
 % Before proving Theorem~\ref{thm:main} for $N_{E}=36$, we prove the following congruence.
 We now prove an analogue of Lemma~\ref{valuationequalsm} 
  \begin{lemma}\label{congruencelevel36}
  For each prime $p \geq 5$  which is inert in $\Q(\sqrt{-3})$ and each integer $m \geq 0$, we have the congruence
  \[
  C_{36}(p^{2m+1}) \equiv (-1)^{m}p^{m}C_{36}(p) \pmod{p^{m+1}}.
  \]
  \end{lemma}
  
  \begin{proof}
  Define
  \[
  \psi_{2}(z)=L(2z)=\sum b(n)q^{6n+4}=q^{-2}+O(q^{4}) \in \Z((q)),
  \]
  \[
  \psi_{3}(z)=L(z)L(2z)-1=\sum c(n) q^{6n+3}=q^{-3} +O(q^{3}) \in \Z((q)).
  \]
  Suppose that $p \equiv 5 \pmod{6}$. Then, we can construct a modular function which is supported on exponents $r$ with $r \equiv 1 \pmod{6}$ of the form
  %The odd primes $p$ which are inert in $\Q(\sqrt{-3})$ satisfy $p \equiv 5\pmod{6}$; for such primes, we can construct modular functions which are supported on exponents $r$ with $r \equiv 1 \pmod{6}$ and of the form
  %There are again modular functions $\psi_{p}$ such that
  \[
  \psi_{p}=q^{-p}-C_{36}(p)q+O(q^{7}).
  \] 
  This implies that
   \[
 \Theta(\psi_{p})=-pq^{-p}-C_{36}(p)q+O(q^{7}) \in M^{\infty}_{2}(36).
 \]
 By Proposition~\ref{proposition4.1}, we also have
 \[
 -G_{36}\sl T_{2}(p)=-pq^{-p}-C_{36}(p)q+O(q^{7}).
 \]
 Thus,
 \begin{equation}\label{thiswasmissinganditreallyannowyedmeh}
 G_{36} \sl T_{2}(p)=-\Theta(\psi_{p}).
 \end{equation}
 The proof follows as in the proof of Lemma~\ref{valuationequalsm}.
  \end{proof}
  We now prove an analogue of Lemma~\ref{lemma3.3}.
  
  \begin{lemma}
  For each prime $p \geq 5$ which is inert in $\Q(\sqrt{-3})$, we have $p \nmid C_{36}(p)$.
  \end{lemma}
  
  \begin{proof}
  Suppose by way of contradiction that $p \equiv 5 \pmod{6}$ has $p \mid C_{36}(p)$. By Proposition~\ref{proposition4.1} and \eqref{thiswasmissinganditreallyannowyedmeh}, we have
  \[
  \Theta(\psi_{p}) \equiv 0 \pmod{p}.
  \]
  This implies that there exist integral coefficients $A(pn)$ such that
  \[
 \psi_{p} \equiv q^{-p}+\sum_{n=1}^{\infty}A(pn)q^{pn} \pmod{p}.
 \]
 A computation in Magma of an integral basis for $M_{2}(36)$ reveals that there is a form $f \in M_{2}(36) \cap \Z[[q]]$ with
 \[
 f=q^{12}+\cdots.
 \]
  %A computation in Magma gives the following basis for $M_{2}(36)$.
  % \begin{align*}
% &f_{1}=q+\cdots, \\
 %&f_{2}=q^{2}+\cdots,\\
 %&f_{3}=q^{3}+\cdots,\\
 %&f_{4}=q^{4}+\cdots,\\
 %&f_{5}=q^{5}+\cdots,\\
 %&f_{6}=q^{6}+\cdots,\\
% &f_{7}=q^{7}+\cdots,\\
 %&f_{8}=q^{8}+\cdots,\\
% &f_{9}=q^{9}+\cdots,\\
 %&f_{10}=q^{10}+\cdots,\\
% &f_{11}=q^{12}+\cdots.
% \end{align*}
% for some integral coefficients $A_{p}(np)$.
 %From a Magma computation, we see that there is a form $f \in M_{2}(36)$ with
 %\[
 %f = q^{6}+\cdots.
% \]
 %This implies that $f^p$ is of the form
% \[
% f^{p} \equiv \sum_{n=12}^{\infty}B_{p}(pn)q^{pn} \equiv q^{12p}+\cdots \pmod{p}.
% \]
 Set
$h_{p}=\psi_{p}f^{p} \in M_{2p}(36).$
 For some integral coefficients $D(pn)$, we have
 \[
 h_{p} \equiv \sum_{n=11}^{\infty}D(pn)q^{pn} \equiv q^{11p}+\cdots \pmod{p},
 \]
 which means that
 \[
 h_{p} \equiv h_{p}\sl U(p) \sl V(p) \pmod{p}.
 \]
 As in Lemma~\ref{lemma3.3}, we conclude that
 %By $(2)$ of Proposition~\ref{Jochnowitz}, we have
 %\[
 %w_p(\bar{h_{p}})=pw_p(\bar{h_{p}\sl U(p)}).
 %\]
 %and $\omega(h_{p}) \equiv 2p \pmod{p-1}$. The fact that $p \mid \omega(h_{p})$ implies that $\omega(h_{p})=2p$, so $\omega(h_{p}\sl U_{p})=2$. 
 there exists a form $h_0 \in M_{2}^{(p)}(36)$ with
 \[
 h_{0} \equiv h_{p}\sl U(p) \equiv q^{11}+O(q^{12}) \pmod{p}.
 \]
 However, an examination of a basis for $M_2(36)$
 % shows that no such form can exist. This gives the result.
 implies that no such form $h_{0}$ exists. The result follows.
  \end{proof}
  
 % \begin{proof}[Proof of Theorem~\ref{thm:main} for $N_{E}=36$]. Lemma $4.2$ and $4.3$ imply \eqref{thm:main part 1}. To prove \eqref{thm:main part 2}, note that Proposition $5.1$ and \eqref{Heckeoperator} give
 %\begin{equation}\label{another3.7}
% \frac{G_{36}\sl U(p^{2m+1})}{C_{36}(p^{2m+1})}-g_{36}=\frac{1}{C_{36}(p^{2m+1})}\(p^{2m+1}%H_{p^{2m+1}}-\sum_{j=1}^{2m+1}p^{j}G_{36}\sl U(p^{2m+1-j})\sl V(p^{j})\).
% \end{equation}
 
% By \eqref{Heckeoperator}, we have
% \[
% G_{36}\sl T_{2}(p^{2m})= \sum_{j=1}^{2m+1}p^{j-1}G_{36}\sl U(p^{2m+1-j})\sl V(p^{j-1}).
% \]
% Given the definition of $G_{36}$, we see that $C_{36}(n)=0$ whenever $n \not \equiv 5 \pmod{6}$. %Proposition~\ref{proposition4.1} then gives that $G_{36} \sl T_{2}(p^{2m})=p^{2m}H_{p^{2m}}$. 
 %Thus,
 %\begin{equation}\label{papakyriakoupolos}
% \sum_{j=1}^{2m+1}p^{j}G_{36}\sl U(p^{2m+1-j})\sl U(p^{2m+1-j})\sl V(p^{j})=p^{2m+1}H_{p^{2m}}\sl %V(p) \equiv 0 \pmod{p^{2m+1}}.
 %\end{equation}
%From \eqref{papakyriakoupolos}, we see that \eqref{thm:main part 1} and \eqref{another3.7} give %\eqref{thm:main part 2}.
%\end{proof}
The proof of Theorem~\ref{thm:main} for $N_{E}=36$ follows as in the proof for $N_{E}=27$.
\section{Proofs of Theorem~\ref{thm:main} for $N_{E}=64$ and $N_{E}=144$}
Theorem~\ref{thm:main} for $N_{E}=64$ and $N_{E}=144$ follow from
the cases $N_{E}=32$ and $N_{E}=36$.
 %properties of twists, the result of Ahlgren and Samart for $N_{E}=32$, and Theorem~\ref{thm:main} for $N_{E}=36$. 
 In particular,if $\chi_{8}$ and $\chi_{12}$ are the Kronecker characters with discriminants $8$ and $12$, we have
\[
g_{64}=g_{32}\otimes\chi_{8},
\ 
G_{64}=G_{32}\otimes\chi_{8},
\]
\[
g_{144}=g_{36}\otimes\chi_{12},
\
G_{144}=G_{36}\otimes\chi_{12}.
\]
The result follows from using, for example, the relationship
\[
\(G_{32} \otimes \chi_{8}\) \sl U(p^{2m+1})= \chi_{8}(p^{2m+1})(g \sl U(p^{2m+1}) \otimes \chi_{8}).
\]
   \section{Acknowledgements}
 
 The author would like to thank Scott Ahlgren for suggesting this project and for advice and guidance for this work. The author would also like to thank  the Graduate College Fellowship program at the University of Illinois at Urbana-Champaign and the Alfred P. Sloan Foundation for their generous research support.  
  \bibliographystyle{amsalpha}

 \bibliography{MoreExamples} 
 
 \end{document}